\documentclass{amsart}
\usepackage{amsmath,amssymb,amsthm,amsfonts,amscd,mathrsfs}
\usepackage[all]{xy}
\usepackage[dvips]{graphicx}

\newcommand{\imm}{\looparrowright}
\newcommand{\be}{\begin{enumerate}}
\newcommand{\ee}{\end{enumerate}}
\newcommand{\R}{\mathbb{R}}
\newcommand{\Z}{\mathbb{Z}}

\newcommand{\co}{\colon\thinspace}

\theoremstyle{plain}
\newtheorem{thm}{Theorem}[section]
\newtheorem{lemma}[thm]{Lemma}

\newtheorem{prop}[thm]{Proposition}

\newtheorem{cor}[thm]{Corollary}

\theoremstyle{definition}
\newtheorem{exam}[thm]{Example}

\begin{document}

\title[On realizing homology classes]{On realizing homology classes by maps of restricted complexity}

\author{Mark Grant}
\author{Andr\'{a}s Sz\H{u}cs}
\address{School of Mathematical Sciences, The University of Nottingham,
University Park, Nottingham, NG7 2RD, UK}
\address{E\"{o}tv\"{o}s Lor\'{a}nd University, P\'{a}zm\'{a}ny P\'{e}ter s\'{e}t\'{a}ny 1/C, 3-206, 1117 Budapest, Hungary}
\email{mark.grant@nottingham.ac.uk}
\email{szucs@math.elte.hu}


\maketitle
\begin{abstract} We show that in every codimension greater than one there exists a mod $2$ homology class in some closed manifold (of sufficiently high dimension) which cannot be realized by an immersion of closed manifolds. The proof gives explicit obstructions (in terms of cohomology operations) for realizability of mod $2$ homology classes by immersions. We also prove the corresponding result in which the word `immersion' is replaced by `map with some restricted set of multi-singularities'.
\end{abstract}

\section{Introduction}

 Let $f\co M^{n-k}\to N^n$ be a continuous map of codimension $k$ between closed manifolds (all manifolds and maps between them are assumed smooth, unless stated otherwise). Then $f$ is said to {\em realize} both the mod $2$ singular homology class $z=f_*[M]\in H_{n-k}(N;\Z_2)$ (where $[M]\in H_{n-k}(M;\Z_2)$ is the fundamental class of the domain manifold) and its Poincar\' e dual cohomology class $x\in H^k(N;\Z_2)$. This paper addresses the following questions. When can a (co)homology class be realized by an immersion? When can a (co)homology class be realized by a map whose complexity is restricted in some way (for instance, by prescribing some finite set of allowed multi-singularity types)?

We first state our results, then put them in historical context. Recall that an {\em immersion} is a smooth map $f\co M^{n-k}\to N^n$ whose differential $df_x\co TM_x\to TN_{f(x)}$ is injective at each point $x\in M$.

\begin{thm}\label{immersions}
For any $k>1$ there exists a closed manifold $N_k$ and cohomology class $x_k\in H^{k}(N_k;\Z_2)$ which cannot be realized by an immersion. The manifold $N_k$ can be chosen to have dimension $4k+3$ if $k$ is even, and $4k+15$ if $k$ is odd.
\end{thm}

The proof of Theorem \ref{immersions} makes use of the following explicit obstructions to realizability by immersions, in terms of stable cohomology operations. (Recall that a {\em stable cohomology operation} is a cohomology operation which commutes with the suspension isomorphism. The definitions of admissible sequence and excess in the context of the mod $2$ Steenrod algebra $\mathcal{A}$ will be recalled in Section 2.)

\begin{thm}\label{mod2}
Let $k>1$. Let $I$ be an admissible sequence of excess $e(I)=k$, and let $Sq^I\in\mathcal{A}$ be the corresponding monomial. If the cohomology class $x\in H^k(N;\Z_2)$ is realizable by an immersion, then $Sq^I(x)$ is the reduction mod $2$ of an integral class.

In particular, if $k$ is even and $\beta(x^2)$ is nonzero (where $\beta$ is the Bockstein associated to reduction mod $2$) then $x$ cannot be realized by an immersion.
\end{thm}

The obstruction $\beta(x^2)$ in the case $k$ even is very natural, in light of Proposition \ref{singset} below: it is the integer cohomology class realized by the singular set of any stable\footnote{Recall that in singularity theory a smooth map $f\co M\to N$ is called {\em stable} if for any sufficiently close smooth map $f'\co M\to N$ there exist diffeomorphisms $g\co M\to M$ and $h\co N\to N$ such that $h^{-1}\circ f' \circ g=f$.} map realizing $x$.

Let $\tau$ be a finite set of codimension $k$ multi-singularities. A {\em multi-singularity} is a finite multiset of stable local singularities; more details will be given in Section 5 below. Recall \cite{R-Sz} that a stable map $f\co M^{n-k}\to N^n$ is called a {\em $\tau$-map} if at each point $y\in N$ the pre-image $f^{-1}(y)\subseteq M$ is finite and the local singularities of $f$ at the pre-image points, counted with multiplicity, form an element of $\tau$.

\begin{thm}\label{taumaps}
Let $k>1$, and let $\tau$ be any finite set of multi-singularities in codimension $k$. Then there exists a closed manifold $N_k$ and cohomology class $x_k\in H^{k}(N_k;\Z_2)$ which cannot be realized by a $\tau$-map.
\end{thm}

Theorems \ref{immersions} and \ref{taumaps} should be contrasted with the well known fact that any one-dimensional cohomology class $x \in H^1(N;\Z_2)$ in a closed manifold is realizable by an embedding of a closed manifold (recall that an {\em embedding} is an immersion that is a homeomorphism onto its image). This follows since $x=w_1(\xi)$ is the mod $2$ Euler class of some line bundle $\xi$ over $N$, and is therefore realized by the zero set of a generic section of $\xi$.

It is somewhat surprising that a result such as Theorem \ref{immersions} has not found its way into the literature before now. Ever since Poincar\'e and the birth of homology, basic questions concerning realization of homology classes by maps from closed manifolds have had a profound effect on the development of Algebraic Topology. Thom showed in his landmark paper \cite{TH} that every mod $2$ homology class in a finite polyhedron can be realized by a continuous map, thus giving an affirmative answer to a problem posed by Steenrod. In its original formulation \cite{Eil49}, Steenrod's question was about realizing integral homology classes by maps from {\em oriented} manifolds, and Thom also gave negative results in this direction, by constructing examples of non-realizable integral homology classes in dimensions $7$ and above.

Thom's method was to reduce Steenrod's problem to the related question concerning realizability of homology classes by embeddings. The key insight which allowed him to solve this problem was that a homology class in the compact manifold $N$ can be realized by a codimension $k$ embedding if and only if its Poincar\' e dual cohomology class is induced from the Thom class by a map from $N/\partial N$ into the Thom space of the universal $k$-dimensional bundle. In other words, the Thom space of the universal $k$-dimensional bundle is the classifying space for codimension $k$ embeddings. One can use this result to find homology classes which cannot be realized by embeddings, in two closely related ways.

 The first is constructive, in that it gives specific obstructions to realizability. Namely, one shows that some expression $P$ involving cup products and cohomology operations vanishes on the Thom class. If the dual of a cohomology class $x$ is to be realizable, that same expression must also vanish on $x$ (this approach was taken by Thom \cite[Chapitre II]{TH}, and more recently by the authors of \cite{BHK} to exhibit new examples of classes non-realizable by submanifolds of certain types).

 The second approach is less constructive, but equally valid. One compares the graded rank of the mod $2$ cohomology of the Thom space of the universal $k$-dimensional bundle with that of the corresponding Eilenberg-Mac Lane space. In high degrees the latter is larger, and so this approach shows that in all dimensions $k>1$ there exists a mod $2$ cohomology class in some closed manifold of sufficiently high dimension which cannot be realized by an embedding (Thom says that this argument, outlined on page 46 of \cite{TH}, was patterned after a remark of J.\-P.\ Serre).

The current paper strengthens both of these approaches, by varying both the choice of classifying space and the maps allowed. To prove our Theorem \ref{immersions} about non-realizability of classes by immersions, we use the fact (due to Wells \cite{We}) that immersions of codimension $k$ are classified by {\em stable}\footnote{Here we use `stable' in the sense of stable homotopy theory, to describe a map which exists after sufficiently many suspensions. We trust that this homonymy will not cause confusion.} maps to the universal Thom space in dimension $k$ (see Section 2). Thus any stable cohomology operation $P$ which vanishes on the universal Thom class must also vanish on any cohomology class realized by an immersion. In Section 3 we apply a result of Browder \cite{Bro74} to show that certain such operations $P$ composed of Steenrod squares and Bockstein operators are nonzero on the fundamental cohomology class. This allows us to construct our smooth examples by `thickening' these universal examples.

 In Section 4, we show that the obstruction $\beta(x^2)$, for an even dimensional cohomology class $x$, is the integer cohomology class realized by the singular set of any stable map realizing $x$. This section also contains a generalisation (Lemma \ref{squares}) of a result due to Thom \cite{Th50} which gives a geometric construction of the Steenrod squares (see also \cite{EG11}).

In Section 5 we recall the construction of the classifying space $X_\tau$ for $\tau$-maps, due to Rim\' anyi and the second author \cite{R-Sz}. In Section 6 we observe that the dimensions of the cohomology groups of $X_\tau$ (viewed as vector spaces over $\Z_2$) grow not faster than those of a finitely generated $\Z_2$-algebra, which leads to the proof of Theorem \ref{taumaps}.

The authors wish to thank Alain Cl\' ement, for including in his thesis \cite{Cle} appendices listing the integral cohomology groups of some $2$-local Eilenberg-Mac Lane spaces, and for a timely and insightful email regarding the Bockstein spectral sequence for these spaces, both of which were vital to the proof of Theorem \ref{immersions}.

{\em For the rest of the paper, all homology and cohomology groups are to be understood with coefficients in $\Z_2$, the ring of integers mod $2$, unless explicitly indicated otherwise.}

\section{Immersions in cohomology classes}

In this section we recall some known facts pertaining to the realization of cohomology classes by immersions, and give the proof of Theorem \ref{mod2}. We fix a closed manifold $N^n$ and a codimension $k>1$.

 If $\zeta$ is a vector bundle on $N$ of rank $k$, we denote by $\mathcal{I}(N;\zeta)$ the bordism group of immersions in $N$ with $\zeta$-structure. The elements of this group are equivalence classes of triples $(M^{n-k},f,v)$ under a suitable bordism relation, where $M$ is a closed manifold, $f\co M^{n-k}\imm N^n$ is an immersion of codimension $k$, and $v\co \nu_f\to \zeta$ is a bundle map isomorphic on fibres. We refer the reader to \cite{EG07} for more details.

Recall that the universal $O(k)$-bundle $\gamma_k\to BO(k)$ is a rank $k$ vector bundle with the property that the inclusion of the nonzero vectors into the total space is homotopically equivalent to the standard inclusion $i\co BO(k-1)\to BO(k)$. We denote by $MO(k)$ the Thom space of $\gamma_k$, and by $U_k\in H^k(BO(k),BO(k-1))\cong\tilde{H}^k(MO(k))$ the universal Thom class. Note that the Thom class is represented by a pointed map $U_k\co MO(k)\to K(\Z_2,k)$.

If $X$ is a based space and $\ell\geq 0$, we denote by $\sigma^\ell\co \tilde{H}^k(X) \stackrel{\cong}{\to} \tilde{H}^{k+\ell}(\Sigma^\ell X)$ the $\ell$-fold reduced suspension isomorphism. We remark that $\sigma^\ell(U_k)\in \tilde{H}^{k+\ell}(\Sigma^\ell MO(k))$ is the Thom class of the direct sum $\gamma_k\oplus\varepsilon^\ell$ of the universal bundle with a trivial bundle.

There is a natural transformation $\Theta\co \mathcal{I}(N;\gamma_k)\to H^k(N)$ which sends the bordism class of an immersion $f\co M^{n-k}\imm N^n$ to the cohomology class it realizes (see \cite[Section 3.2]{EG07}). Representing these functors homotopically, we have the following diagram
\begin{equation}
\xymatrix{
\mathcal{I}(N;\gamma_k) \ar[r]^-{\Theta} \ar[d]_{\cong} & H^k(N) \ar[d]_{\cong}\\
[\Sigma^\ell N_+,\Sigma^\ell MO(k) ]_{\ell\gg 0} \ar[r] &  [\Sigma^\ell N_+,K(\Z_2,k+\ell)],
}
\end{equation}
where square brackets denote pointed homotopy classes of maps. The right hand isomorphism comes from the isomorphism $H^k(N)\cong \tilde{H}^k(N_+)$, where $N_+$ denotes $N$ with a disjoint base-point, composed with the suspension isomorphism and the classification of reduced cohomology as maps into Eilenberg-Mac Lane spaces. The left-hand isomorphism identifying $\mathcal{I}(N;\gamma_k)$ with a stable homotopy group is the Pontrjagin-Thom-Wells theorem \cite{We}. The lower horizontal arrow is induced by the map $\Sigma^\ell MO(k)\to K(\Z_2,k+\ell)$ which represents $\sigma^\ell(U_k)\in \tilde{H}^{k+\ell}(\Sigma^\ell MO(k))$.

From these remarks it is trivial to deduce the following slight extension of Thom's fundamental result on realizability of cohomology classes by embeddings \cite[Th\'{e}or\`{e}me II.1]{TH}.
\begin{prop}\label{Wells}
The cohomology class $x\in H^k(N)$ is realizable by an immersion if, and only if, there exists a map $F\co \Sigma^\ell N_+\to \Sigma^\ell MO(k)$, where $\ell$ is large, such that
\begin{equation}
\sigma^\ell(x) = F^* \sigma^\ell (U_k) \in \tilde{H}^{k+\ell}(\Sigma^\ell N_+).
\end{equation}
\end{prop}

\begin{cor}\label{stable}
Let $P$ be a stable cohomology operation such that $P(U_k)=0$. If $x\in H^k(N)$ is realizable by an immersion then $P(x)= 0$ also.
\end{cor}

The stable operations in mod $2$ cohomology make up the mod $2$ Steenrod algebra $\mathcal{A}$, which consists of non-commuting polynomials in the Steenrod squares $Sq^i\co H^*(-)\to H^{*+i}(-)$, subject to the Adem relations. We recall some standard notation (as found for example in \cite{MT}). We denote by $Sq^I=Sq^{i_1}Sq^{i_2}\cdots Sq^{i_r}$ the Steenrod operation corresponding to the multi-index $I=(i_1,\ldots, i_r)$ of non-negative integers. Such a sequence is called {\em admissible} if $i_j\geq 2 i_{j+1}$ for $j = 1,\ldots , r-1$. Its {\em dimension} is $|I|=\sum_j i_j$, and its {\em excess} is $e(I)=\sum_j i_j - 2i_{j+1} = 2i_1 - |I|$. The significance of these definitions comes from Serre's Theorem \cite{Ser53} describing the mod $2$ cohomology of Eilenberg-Mac Lane spaces as a polynomial algebra
  \begin{equation}\label{Serre}
  H^*\big(K(\Z_2,k)\big) = \Z_2\left[Sq^I(\iota_k)\mid I\mbox{ admissible, }e(I)<k\right].
  \end{equation}
  Here $\iota_k\in H^k\left( K(\Z_2,k)\right)\cong \mathrm{Hom}(\Z_2,\Z_2)$ is the fundamental class corresponding to the identity homomorphism. We remark that if $x$ is any cohomology class in degree $k$ and $I=(i_1,i_2,\ldots ,i_r)$ has excess $k$, then $Sq^I(x) = \left[Sq^J(x)\right]^2$, where $J=(i_2,\ldots , i_r)$.

  As it turns out, Steenrod operations alone are insufficient for finding obstructions to realizability by immersions - we also have to consider the Bockstein coboundary associated to the short exact coefficient sequence $0\to \Z \to \Z \to \Z_2 \to 0$. This gives rise to the long exact sequence in cohomology
\begin{equation}\label{exact}
\xymatrix{
\cdots \to  H^n(X,A;\Z) \ar[r]^-{\cdot 2} & H^n(X,A;\Z) \ar[r]^-{\rho} & H^n(X,A) \ar[r]^-\beta & H^{n+1}(X,A;\Z)  \to\cdots
}
\end{equation}
for any space pair $(X,A)$, where $\rho$ denotes reduction mod $2$. Note that $\beta$ is a stable operation of degree one, with the property that $y\in H^*(X,A)$ is the reduction mod $2$ of an integral class if and only if $\beta(y)=0$. We also remark that $\rho\circ\beta = Sq^1\co H^*(X,A)\to H^{*+1}(X,A)$.

\begin{proof}[Proof of Theorem \ref{mod2}]
We aim to show that $\beta Sq^I(U_k)=0$ whenever $I$ has excess $k$. Theorem \ref{mod2} will then follow directly from Corollary \ref{stable} and the exact sequence (\ref{exact}).

We first claim that the long exact sequence in integral cohomology of the pair $(BO(k),BO(k-1))$ splits into short exact sequences
\begin{equation}
\xymatrix{
0 \to H^*(BO(k),BO(k-1);\Z) \ar[r]^-{j^*} & H^*(BO(k);\Z) \ar[r]^-{i^\ast} & H^*(BO(k-1);\Z)  \to 0.
}
\end{equation}
To see this, it suffices to check that $i^\ast$ is surjective in every positive dimension. This follows from the description of the integer cohomology rings of the $BO(k)$, given by Brown in \cite[Theorem 1.6]{Bro82}. Indeed, $H^*(BO(k);\Z)$ is a quotient of a polynomial ring generated by the Pontrjagin classes $p_i(\gamma_k)$ and the Bockstein images of certain monomials in the Stiefel-Whitney classes $w_i(\gamma_k)$. Since $\gamma_{k-1}\cong i^*\gamma_k$, and $i^*$ is a ring homomorphism commuting with the Bockstein operator, it follows that each polynomial generator in $H^*(BO(k-1);\Z)$ is the image under $i^*$ of a generator of $H^*(BO(k);\Z)$, and the claim is proved.

It is well known that all torsion in $H^*(BO(k);\Z)$ is of order $2$ (see \cite[Lemma 2.4]{Bro82} for example), and it follows that the same is true of $H^*(BO(k),BO(k-1);\Z)$. It is also well known \cite[Lemma 2.2]{Bro82}, and easy to see using (\ref{exact}), that if all torsion in $H^{n+1}(X,A;\Z)$ is of order $2$ and $y\in H^n(X,A)$, then $\beta(y)=0$ if and only if $Sq^1(y)=0$. Hence $\beta Sq^I(U_k)=0$ if and only if $Sq^1 Sq^I(U_k)=0$. The latter is true since $Sq^1$ is a derivation and $Sq^I(U_k) = \left[Sq^J(U_k)\right]^2$.

In particular, if $I=(k)$ then $\beta Sq^k(U_k) = \beta (U_k^2) = 0$. Note that if $k = 2m +1 $ is odd then the operation $\beta Sq^{2m+1}$ is identically zero, as follows from the relation $Sq^{2m+1}=Sq^1 Sq^{2m}$ and the exact sequence (\ref{exact}).
\end{proof}

\section{Cohomology of Eilenberg-Mac Lane spaces}

Our aim in this section is to prove Theorem \ref{immersions}, by exhibiting cohomology classes $x_k\in H^k(N_k)$ and operations $Sq^I$ with $e(I)=k$ such that $\beta Sq^I(x_k)\neq 0$. If we are to have any chance of success, we must be able to do this for the universal example given by the fundamental class $\iota_k\in H^k(K(\Z_2,k))$. Hence we need to recall some facts about the mod $2$ Bockstein spectral sequence for the spaces $K(\Z_2,k)$, all of which can be found in the paper of Browder \cite{Bro74}.

The long exact sequence (\ref{exact}) (when $A=\emptyset$) rolls up into an exact couple
\begin{equation}
\xymatrix{
H^*(X;\Z) \ar[rr]^-{\cdot 2} & & H^*(X;\Z) \ar[ld]^-{\rho} \\
 & H^*(X) \ar[ul]^-{\beta} &
 }
 \end{equation}
 which gives rise to a spectral sequence in the usual way. This is the well known {\em mod $2$ Bockstein spectral sequence} \cite{Bro74}, \cite{MT}. It has first page $E_{(1)}^*(X)=H^*(X)$ and first differential $d_{(1)}=\rho\circ\beta = Sq^1$. The $E_{(2)}^*(X)$ page is therefore the cohomology of $H^*(X)$ with respect to the derivation $Sq^1$. Given an element $y\in H^*(X)$ with $Sq^1(y)=0$, we denote by $\lbrace y\rbrace\in E_{(2)}^*(X)$ the class represented by $y$. The following lemma is then obvious from the construction of the spectral sequence.

\begin{lemma}\label{diff}
Let $y\in H^*(X)$ with $Sq^1(y)=0$. If $d_{(2)}\lbrace y\rbrace \in E^{\ast +1}_{(2)}(X)$ is nonzero, then so is $\beta(y)\in H^{*+1}(X;\Z)$.
\end{lemma}

We wish to apply this lemma with $y\in H^*\big(K(\Z_2,k)\big)$. Browder gave a detailed description of the mod $p$ Bockstein spectral sequence of $K(\Z_{p^s},k)$ for all primes $p$ and $s\geq 1$. Here we state only the small part of his results (in the case $p=2$ and $s=1$) needed for our purpose.

\begin{thm}[{Browder, \cite[Theorem 5.5]{Bro74}}]\label{E2}
The second page of the mod $2$ Bockstein spectral sequence for $K(\Z_2,k)$ is a polynomial algebra over $\Z_2$,
\begin{equation}
E_{(2)}^*\big( K(\Z_2,k)\big) = \Z_2\left[ \lbrace G^2\rbrace, d\lbrace G^2\rbrace\right].
\end{equation}
Here $G$ runs over the even dimensional generators of $H^*\big( K(\Z_2,k)\big)$ (excluding the generator $Sq^1(\iota_k)$ in the case $k$ odd), and $d=d_{(2)}$ is the second differential. Hence $G=Sq^J(\iota_k)$ where $J=(j_1,\ldots , j_s)$ is admissible of excess less than $k$ such that $|J|+k$ is even and $j_1\neq 1$.
\end{thm}
\begin{proof}[Proof of Theorem \ref{immersions}]
Let $J=(j_1,\ldots, j_s)$ be an admissible sequence of excess less than $k$ such that $j_1\neq 1$ and $|J|+k$ is even. Then the sequence $I=(|J|+k,j_1,\ldots , j_s)$ is admissible of excess $k$, and it follows from Theorem \ref{E2} and Lemma \ref{diff} that the element $$y= Sq^I(\iota_k) = \left[Sq^J(\iota_k)\right]^2 $$ is not the reduction of an integral class.

For example, if $k$ is even we can set $J=(0)$ and then $Sq^I(\iota_k)=Sq^k(\iota_k)=\iota_k^2$.

If $k$ is odd we can set $J=(2,1)$ and then $Sq^I(\iota_k) = Sq^{k+3}Sq^2Sq^1(\iota_k)$.

Now we can construct smooth examples by the method of thickenings, as follows. If $k$ is even, take the $(2k+2)$-skeleton $K=K^{(2k+2)}\subseteq K(\Z_2,k)$, whose cohomology is isomorphic to that of $K(\Z_2,k)$ up to degree $2k+1$, and choose an embedding $K\hookrightarrow\R^{n+1}$ in some high-dimensional euclidean space. This embedding has a regular neighbourhood $W\supset K$, homotopy equivalent to $K$, whose boundary $\partial W$ is a closed smooth $n$-manifold. By an easy argument involving Lefschetz duality, one sees that the map
\begin{equation}
H^{2k+1}(K;\Z)\cong H^{2k+1}(W;\Z)\to H^{2k+1}(\partial W;\Z)
\end{equation}
is injective as long as $n\geq 4k +3$. Thus the image $x_k$ of $\iota_k$ in $H^k(\partial W)$ has the property that $\beta(x_k^2)\neq 0$, and so cannot be realized by an immersion. The smallest dimensional example obtained by this method is a $2$-dimensional class in a closed $11$-manifold.

If $k$ is odd, the same technique produces a class $x_k$ in a $(4k+15)$-dimensional manifold such that $\beta Sq^{k+3}Sq^2 Sq^1(x_k)$ is nonzero, and hence $x_k$ cannot be realized by an immersion. The smallest dimensional example obtained by this method is a $3$-dimensional class in a closed $27$-manifold.
\end{proof}

\section{Interpretation of the obstruction $\beta(x^2)$}

When $k$ is even, the obstruction $\beta(x^2)\in H^{2k+1}(N;\Z)$ to a class $x\in H^k(N)$ being realizable by an immersion has a very natural interpretation: it is the integer class realized by the singular set of any stable map realizing $x$. In order to make this precise, we need several preliminaries.

Let $f\co (M,\partial M)\to (N,\partial N)$ be a codimension $k$ map between compact manifolds, where $\partial M$ and $f(M)\cap \partial N$ are assumed empty. Denote the virtual normal bundle by $\nu_f = f^*TN-TM$. It is well-known that, for any local system $\mathscr{L}$ of abelian groups on $N$, the map $f$ induces a {\em Gysin homomorphism}
\begin{equation}
f_!\co H^*(M;f^*\mathscr{L}\otimes\Z_f)\to H^{*+k}(N,\partial N;\mathscr{L}),
\end{equation}
where $\Z_f$ denotes the local system of integers twisted by $w_1(\nu_f)$ (for details, see \cite{OSS}). In particular, when $\mathscr{L}$ is the trivial system with group $\Z$ or $\Z_2$, we get Gysin homomorphisms
\begin{equation}
f_!\co H^*(M;\Z_f)\to H^{*+k}(N,\partial N;\Z)\qquad\mbox{and}\qquad f_!\co H^*(M)\to H^{*+k}(N,\partial N).
\end{equation}
With these notations, the mod $2$ cohomology class $x\in H^k(N,\partial N)$ realized by $f$ is just $f_!(1)$, where $1\in H^0(M)$ is the unit class.

A well-known result of Thom \cite{Th50} (see also \cite{EG11}) states that $f_!\big(w_i(\nu_f)\big)=Sq^i\big(f_!(1)\big)$ when $f\co M\imm N$ is an immersion with $M$ closed. The following Lemma generalises Thom's result to singular maps and integer coefficients.
\begin{lemma}\label{squares}
Let $x\in H^k(N)$ be a cohomology class realized by a map of closed manifolds $f\co M^{n-k}\to N^n$. Let $W_r(\nu_f)$ denote the $r$-th Stiefel-Whitney class of the virtual normal bundle of $f$ (with coefficents in $\Z_2$ for $r$ even, and twisted coefficients in $\Z_f$ for $r$ odd). Then
\begin{equation}
f_!\big(W_{2i}(\nu_f)\big) = Sq^{2i}(x)\quad\mbox{and}\quad f_!\big(W_{2i+1}(\nu_f)\big)=\beta Sq^{2i}(x).
\end{equation}
\end{lemma}
\begin{proof}
We use stabilisation to reduce to the embedded case. Let $e\co (N,\emptyset)\to (N\times D^q,N\times S^{q-1})$ be the standard embedding, where $D^q$ is a large dimensional disk with boundary $S^{q-1}$. It is well known that $e_!$ can be identified with the $q$-fold suspension isomorphism, for any coefficients, and in particular $e_!$ commutes with stable cohomology operations. The composition $e\circ f$ is homotopic to an embedding $g\co (M,\emptyset)\hookrightarrow (N\times D^q,N\times S^{q-1})$ with image in the interior of $N\times D^q$ and such that $W(\nu_g)=W(\nu_f)$. Since $e_!\circ f_! = g_!$ the claimed equalities are equivalent to
\begin{equation}
g_!\big(W_{2i}(\nu_g)\big) = Sq^{2i}\big(g_!(1)\big)\quad\mbox{and}\quad g_!\big(W_{2i+1}(\nu_g)\big)=\beta Sq^{2i}\big(g_!(1)\big).
\end{equation}

 We will obtain these equalities by showing that they hold for the universal embedding $g_u\co BO(k+q)\to MO(k+q)$ of codimension $k+q$, then applying naturality arguments. (Here we are abusing notation slightly, since the universal embedding is really a map $g_u\co (BO(k+q),\emptyset)\to (D\gamma_{k+q},S\gamma_{k+q})$ which is obtained as a limit of codimension $k+q$ embeddings of finite dimensional manifolds.)

Let $W_{2i+1}=W_{2i+1}(\gamma_{k+q})$ be the universal Stiefel-Whitney class. We wish to show that
\begin{equation}\label{universal}
{g_u}_!(W_{2i+1})=\beta Sq^{2i}(U_{k+q})\in \tilde{H}^{k+q+2i+1}\big(MO(k+q);\Z\big).
\end{equation}
By the results stated in the proof of Theorem \ref{mod2} at the end of Section 2, we have that $\tilde{H}^{*}(MO(k+q);\Z)\subseteq H^*(BO(k+q);\Z)$, and the latter group is $2$-torsion in all dimensions not divisible by $4$. Hence by choosing $q$ so that $k+q+2i+1\not\equiv 0 (\mathrm{mod} 4)$ we can ensure that the mod $2$ reduction
\begin{equation}
\rho\co \tilde{H}^{k+q+2i+1}\big(MO(k+q);\Z\big)\to \tilde{H}^{k+q+2i+1}\big(MO(k+q)\big)
\end{equation}
is injective. Using that the reduction mod $2$ of ${g_u}_!$ can be identified with the universal Thom isomorphism, and the identities $\rho\circ\beta = Sq^1$ and $Sq^1 Sq^{2i} = Sq^{2i+1}$, we see that the mod $2$ reduction of equality (\ref{universal}) becomes the well-known identity $w_{2i+1}\cdot U_{k+q} = Sq^{2i+1}(U_{k+q})$. Hence equality (\ref{universal}) holds.

The equality ${g_u}_!(W_{2i})= Sq^{2i}(U_{k+q})$ is immediate, since $W_{2i}\cdot U_{k+q} = Sq^{2i}(U_{k+q})$.

Let $T$ denote a closed tubular neighbourhood of $g$ in the interior of $N\times D^q$. By universality we obtain a transversal pullback square
\begin{equation}
\xymatrix{
(M,\emptyset) \ar[d]^{g} \ar[r]^-{\nu_g} & \big(BO(k+q),\emptyset\big) \ar[d]^{g_u} \\
(T,\partial T) \ar[r]^-j & (D\gamma_{k+q},S\gamma_{k+q})
}
\end{equation}
from which it follows that in $H^*(T,\partial T;\Z)$ we have
\begin{align*}
g_!\big(W_{2i+1}(\nu_g)\big)   & = g_! \circ \nu_g^*(W_{2i+1}) \\
                       & = j^* \circ {g_u}_!(W_{2i+1}) \\
                       & = j^* \beta Sq^{2i}(U_{k+q}) \\
                       & = \beta Sq^{2i}\big(j^*(U_{k+q})\big) \\
                       & = \beta Sq^{2i}\big(g_!(1)\big).
\end{align*}
Applying the excision isomorphism $H^*(T,\partial T;\Z)\cong H^*(N\times D^q,N\times D^q\setminus\operatorname{int}(T))$ composed with the restriction $H^*(N\times D^q,N\times D^q\setminus\operatorname{int}(T))\to H^*(N\times D^q, N\times S^{q-1})$ yields the desired equality in the odd case. An analogous argument applies to the even case also.
\end{proof}

Now let $x\in H^k(N)$, with $k$ even. Suppose that $x$ can be realized by a stable map $f\co M^{n-k}\to N^n$ of closed manifolds. The {\em singular set} of $f$ is defined as
\begin{equation}
\Sigma = \lbrace x\in M \mid \mathrm{rank}(df_x\co TM_x\to TN_{f(x)}) < n-k\rbrace \subseteq M.
\end{equation}
It is the closure in $M$ of the simplest singularity stratum $\tilde\Sigma = \Sigma^{1,0}\subseteq M$ of Whitney umbrella points. This top stratum $\tilde{\Sigma}$ is an open dense subset of $\Sigma$, and has codimension $k+1$ in $M$. All other strata have codimension at least $k+4$ in $M$. It is well known that $\Sigma$ carries a fundamental class
\begin{equation}
[\Sigma]\in H_{n-2k-1}(\Sigma;\Z_{\tilde{\Sigma}})
\end{equation}
(the coefficients are twisted according to the tangent bundle of $\tilde\Sigma$; the other strata do not affect the orientability). Denote by $\imath\co \Sigma\to M$ the inclusion, and by $\bar{f} = f\circ \imath\co \Sigma \to N$ the restriction of $f$ to its singular set.
\begin{prop}\label{singset}
 Let $x\in H^k(N)$ be realized by a stable map $f\co M^{n-k}\to N^n$ of closed manifolds, where $k$ is even. Then $\beta(x^2)\in H^{2k+1}(N;\Z)$ is the cohomology class in $N$ realized by the singular set of $f$. In other words, $\beta(x^2)$ is the Poincar\'e dual of
 \begin{equation}
 \bar{f}_*[\Sigma] \in H_{n-2k-1}(N;\Z_N).
 \end{equation}
\end{prop}
\begin{proof}
Let $\jmath\co \tilde{\Sigma}\to M$ be the inclusion of the set of Whitney umbrella points, and denote by $\tilde{f} = f\circ\jmath$ the restriction of $f$ to $\tilde{\Sigma}$. Hence $\tilde{f}$ is an immersion, whose normal bundle $\nu_{\tilde{f}}\cong\nu_{\jmath}\oplus \jmath^*\nu_f$ we claim is oriented, regardless of the orientablility of $M$, $N$ and $\nu_f$. This follows from arguments in \cite{Sz} (see also \cite{R-Sz} for a more sophisticated approach). The umbrella points have normal form
\begin{equation}
(x,y)\mapsto (u,v,w) = (x^2,y,xy),\qquad (x,u\in\R^1,\quad y,v,w\in\R^k).
\end{equation}
In a tubular neighbourhood of $\tilde{\Sigma}$, the normal co-ordinates can be chosen so that the transition maps have the form
\begin{equation}
(x,y)\mapsto (\varepsilon x, A\cdot y),\qquad \varepsilon = \pm 1,\quad A\in \mathrm{GL}(k,\R).
\end{equation}
The corresponding transition maps in the target will be
\begin{equation}
(u,v,w)\mapsto (u, A\cdot v,\varepsilon A\cdot w).
\end{equation}
Since $k$ is even, it is clear that the transition matrices in the target all have positive determinant.

It follows that $w_1(\nu_{\jmath}) = \jmath^*w_1(\nu_f)$, and hence that $w_1(\tilde{\Sigma}) = \jmath^*(w_1(M) + w_1(\nu_f))$. Hence our fundamental class $[\Sigma]\in H_{n-2k-1}(\Sigma;\Z_{\tilde{\Sigma}})$ can be viewed as an element in the group in the upper-left of the commutative diagram,
\begin{equation}\label{diagram}
\xymatrix{
H_{n-2k-1}(\Sigma;\imath^*(\Z_f\otimes\Z_M)) \ar[r]^-{\imath_*} & H_{n-2k-1}(M;\Z_f\otimes\Z_M)  \ar[r]^-{f_*} & H_{n-2k-1}(N;\Z_N)  \\
                                                      & H^{k+1}(M;\Z_f) \ar[u]^{\cap [M]} \ar[r]^-{f_!} & H^{2k+1}(N;\Z) \ar[u]^{\cap [N]},
}
\end{equation}
in which the vertical arrows are Poincar\'e duality isomorphisms. Together with Lemma \ref{squares}, diagram (\ref{diagram}) reduces the proof of the Proposition to the statement that $W_{k+1}(\nu_f)\cap[M] = \imath_*[\Sigma]$. That is, {\em the Poincar\'e dual in $M$ of the homology class realized $\Sigma$  is the $(k+1)$-st Stiefel-Whitney class of the virtual normal bundle of $f$}.

We can see this as follows. Let $\xi = \hom(TM,f^*TN)$ be the vector bundle over $M$ whose fibre over a point $x\in M$ is the space of linear maps $TM_x\to TN_{f(x)}$. Let $S_x\in \xi_x$ be the subspace consisting of those maps with non-trivial kernel, and let $S=\bigcup_{x\in M} S_x$. Note that $S\subseteq E(\xi)$ is a stratified subset of the total space of $\xi$, and that by restricting to its complement we obtain a fibration $\zeta\co \big(E(\xi) \setminus S\big)\to M$, with fibre $V_{n-k}(\R^n)$ the Stiefel manifold of $(n-k)$-frames in $\R^n$, associated to $\xi$. The obstruction class of this $\zeta$ is precisely the twisted integer Stiefel-Whitney class $W_{k+1}(\nu_f)$.

On the other hand, the differential of $f$ defines a generic section $df\co M\to E(\xi)$ of $\xi$, and it is clear that $\Sigma = df^{-1}(S)$. Hence the homology class realized by $\Sigma$ is dual to the obstruction class of $\zeta$, by standard obstruction theoretic arguments.
\end{proof}

\section{Classifying space for $\tau$-maps}

In this section we recall some details of the classifying space $X_\tau$ for $\tau$-maps, constructed in \cite{R-Sz}, which will be needed for the proof of Theorem \ref{taumaps} in the next section.

We first recall some terminology from singularity theory. Fix a codimension $k>1$. A {\em multi-singularity} of codimension $k$ is defined to be a multiset of stable local singularities of codimension $k$. Here a {\em local singularity} of codimension $k$ is an equivalence class of map germs $\eta\co(\R^{n-k},\mathbf{0})\to (\R^{n},\mathbf{0})$, under the relations generated by {\em suspension} and $A$-equivalence. We refer the reader to \cite{R-Sz} for more details. By abuse of notation, we will denote the local singularity $[\eta]$ by $\eta$.

   Let $\tau$ be a finite set of multi-singularities of codimension $k$. Recall that a {\em $\tau$-map} is a stable map $f\co M^{n-k}\to N^n$ such that at each point $y\in N$ the pre-image $f^{-1}(y)\subseteq M$ is finite, and the multiset of local singularities of $f$ at the points in the pre-image is an element of $\tau$. In \cite{R-Sz} the authors construct a classifying space $X_\tau$ for $\tau$-maps. This is a pointed space with the property that the cobordism classes of $\tau$-maps $f\co M^{n-k}\to N^n$, with $M$ closed, are in one-to-one correspondence with the set of pointed homotopy classes $[N_+,X_\tau]$ when $N$ is a closed manifold.

    The space $X_\tau$ is $(k-1)$-connected, and its lowest dimensional nonzero cohomology group has a single generator $U_\tau\in H^k(X_\tau)$, represented by a map $U_\tau\co X_\tau\to K(\Z_2,k)$, which plays the role of the Thom class in the following sense. Whenever $f\co M^{n-k}\to N^n$ is a $\tau$-map, and $F\co N_+\to X_\tau$ is the corresponding map into the classifying space, then $F^*(U_\tau)\in H^k(N)$ is the cohomology class realized by $f$. Thus we have the following result, analogous to Proposition \ref{Wells}.

\begin{prop}\label{Wellstau}
The cohomology class $x\in H^k(N)$ is realizable by a $\tau$-map if, and only if, there exists a map $F\co N_+\to X_\tau$ such that $x=F^*(U_\tau)$.
\end{prop}

We omit the details of the proof. The result can be deduced from Proposition \ref{Wells} and the fact that the singularity set of a codimension $k$ map has
codimension $(k+1)$ in the source and $(2k+1)$ in the target.

The space $X_\tau$ is constructed by inductively glueing together disc bundles of vector bundles, one for each multi-singularity in $\tau$. This requires a well-ordering on $\tau$, which is obtained by extension of the following partial ordering on the set of codimension $k$ multi-singularities. We can denote a typical multi-singularity by $\theta = m_1\eta_1 + \cdots + m_s\eta_s$ where $m_1,\ldots , m_s$ are natural numbers and $\eta_1,\ldots,\eta_s$ are local singularities. Let $f\co M^{n-k}\to N^n$ be a stable mapping. A point $y\in N$ is called a {\em $\theta$-point} if the pre-image $f^{-1}(y)$ is finite, and the multiset of local singularities of $f$ at the points in the pre-image is precisely $\theta$. We then say that $\theta<\theta'$, where $\theta'$ is some other multi-singularity, if in any neighbourhood of a $\theta'$-point there is always to be found a $\theta$-point.

\begin{exam}
Using the Thom-Boardman classification of singularities, a regular $r$-tuple point would be a $r\Sigma^0$-point, and a Whitney umbrella would be a $\Sigma^{1,0}$-point. We then have $2\Sigma^0<3\Sigma^0$ and $2\Sigma^0<\Sigma^{1,0}$, but $3\Sigma^0$ is not comparable with $\Sigma^{1,0}$.
\end{exam}

We next describe the vector bundle associated to a given multi-singularity $\theta= m_1\eta_1 + \cdots + m_s\eta_s$ in $\tau$. Each local singularity $\eta_i$ has an associated vector bundle $\tilde\xi_i$, which is the universal normal bundle of the (simple) $\eta_i$ stratum in the target. The base space of $\tilde\xi_i$ is the classifying space $BG_i$, where $G_i$ is a compact Lie group. In fact, $G_i$ is the maximal compact subgroup (in the sense of
J\"anich \cite{Jan} and Wall \cite{Wal}) of the group of symmetries of $\eta_i$.

Let $r$ be a natural number, and let $S_r$ be the symmetric group on $r$ letters. If $\xi\to B$ is a vector bundle, the {\em $r$-th extended power} of $\xi$ is the bundle
\begin{equation}\label{extpower}
\mathcal{S}_r(\xi) = ES_r\times_{S_r} \left(\xi\times\cdots \times \xi\right)
 \end{equation}
    which is the Borel construction applied to the $r$-fold Cartesian power of $\xi$, with $S_r$-action given by permutation of the factors (here $ES_r$ denotes a free contractible $S_r$-space). It has Thom space $D_rT\xi$, where, for a pointed space $X$ and $r\geq 1$, the notation $D_rX$ refers to the {\em $r$-adic construction}
\begin{equation}
D_rX = (ES_r)_+ \wedge_{S_r}X\wedge\cdots\wedge X.
\end{equation}
The vector bundle associated to $\theta = m_1\eta_1 + \cdots + m_s\eta_s$ is
\begin{equation}
\tilde{\xi}_\theta = \mathcal{S}_{m_1}(\tilde\xi_1)\times\cdots\times\mathcal{S}_{m_s}(\tilde\xi_s).
\end{equation}
Note that the Thom space of this bundle is the smash product
\begin{equation}\label{Thomspace}
T\tilde{\xi}_\theta = \bigwedge_{i=1}^s D_{m_i}T\tilde\xi_i.
\end{equation}

To construct the space $X_\tau$ one starts with a point, then at each stage glues the disc bundle $D\tilde{\xi}_\theta$ of the next multi-singularity $\theta$ appearing in $\tau$ to the space obtained at the previous stage, by a map defined on the sphere bundle $S\tilde{\xi}_\theta$.

\section{Proof of Theorem \ref{taumaps}}

Let $\tau$ be a finite set of multi-singularities. By the construction outlined in the previous section, the classifying space $X_\tau$ comes with a natural filtration, and an associated spectral sequence converging to the cohomology ring $H^*(X_\tau)$. In this section, by examining the $E_2$-term of this spectral sequence, we prove the following.

\begin{prop}\label{Hilbert}
The dimensions of the cohomology groups of $X_\tau$ (viewed as vector spaces over $\Z_2$) grow not faster than those of a finitely generated $\Z_2$-algebra.
\end{prop}

 Assuming Proposition \ref{Hilbert} for a moment, we now give the proof of Theorem \ref{taumaps}.

  \begin{proof}[Proof of Theorem \ref{taumaps}]
  The cohomology of the Eilenberg-Mac Lane space $K(\Z_2,k)$ (where $k >1$) is a polynomial ring on infinitely many generators, by Serre's Theorem \cite{Ser53}. Hence by Proposition \ref{Hilbert}, the finite group $H^N(K(\Z_2,k))$ has more elements than the finite group $H^N(X_\tau)$, for some $N$ sufficiently large. It follows that there does not exist a pointed map $F\co K(\Z_2,k)\to X_\tau$ with $F^*(U_\tau)=\iota_k\in H^k(K(\Z_2,k))$, where $U_\tau\in H^k(X_\tau)$ is the Thom class. Indeed, such a map would have $U_\tau\circ F\co K(\Z_2,k)\to K(\Z_2,k)$ homotopic to the identity, which would imply injectivity of the induced map $U_\tau^*\co H^N(K(\Z_2,k))\to H^N(X_\tau)$.

   By embedding a suitable skeleton of $K(\Z_2,k)$ in some Euclidean space, and taking the boundary of a regular neighbourhood (as in the proof of Theorem \ref{mod2}), we obtain a closed manifold $N_k$ and a cohomology class $x_k\in H^{k}(N_k;\Z_2)$ which is not a pullback of the Thom class $U_\tau$, and hence by Proposition \ref{Wellstau} cannot be realized by a $\tau$-map.
  \end{proof}

 \begin{proof}[Proof of Proposition \ref{Hilbert}]
The cohomology groups of $X_\tau$ have dimensions bounded above by those of the $E_2$-term of the spectral sequence converging to $H^*(X_\tau)$. From the construction of $X_\tau$ (attaching successive disc bundles of the vector bundles associated to the multi-singularities in $\tau$, along maps defined on their boundary sphere bundles) it follows that the sum of the entries in the $E_2$-term is the sum of the cohomologies of the Thom spaces associated to multi-singularities in $\tau$,
\begin{equation}
\bigoplus_{\theta\in\tau}\tilde{H}^*(T\tilde{\xi}_\theta).
\end{equation}
Since $\tau$ is finite, there are finitely many natural numbers $\alpha_1,\ldots ,\alpha_s$ and local singularities $\eta_1,\ldots ,\eta_s$ such that every multi-singularity $\theta\in\tau$ can be written in the form $\theta = m_1\eta_1 + \cdots + m_s\eta_s$, where the $m_i$ are non-negative integers with $m_i\leq \alpha_i$. By the description (\ref{Thomspace}) and the K\" unneth Theorem, it follows that the dimensions of the $E_2$-term are bounded above by those of
\begin{equation}
\bigotimes_{i=1}^s \bigoplus_{m_i=1}^{\alpha_i} \tilde{H}^*(D_{m_i}T\tilde{\xi}_i).
\end{equation}
Since a finite tensor product or finite direct sum of finitely generated algebras is again finitely generated, the proof will be complete if we can show that the dimensions of $\tilde{H}^*(D_{r} T\tilde{\xi}_i)$ grow not faster than those of a finitely generated algebra, for any natural number $r$ and $i=1,\ldots , s$.

Recall that if $\xi \to B$ is a vector bundle, then $D_rT\xi$ is the Thom space of the $r$-th extended power $\mathcal{S}_r(\xi)$ of $\xi$ (defined at (\ref{extpower}) above). Hence there is a Thom isomorphism
\begin{equation}
\tilde{H}^*(D_{r}T\tilde{\xi}_i) \cong   H^{*-d}\left(\mathcal{S}_r(BG_i)\right),
\end{equation}
where $d = \dim\mathcal{S}_r(\tilde{\xi}_i)$ and $G_i$ is the maximal compact subgroup of the local singularity $\eta_i$.

The space $\mathcal{S}_r(BG_i)$ is the classifying space of the wreath product $S_r\int G_i$, which is itself a compact Lie group. To finish the proof, we quote a result of Venkov \cite{Ven} (see also Quillen \cite[Corollary 2.2]{Q}), which states that if $G$ is a compact Lie group, then the cohomology algebra $H^*(BG)$ is finitely generated. Thus it follows that (up to a shift in degrees) the growth of the cohomology groups $\tilde{H}^*(D_{r} T\tilde{\xi}_i)$ is as claimed.
\end{proof}

\end{document}